\documentclass[12pt]{amsart}
\usepackage{fullpage}
\usepackage{amssymb,amsthm,amsmath}
\usepackage[all]{xy}
\usepackage{cite}

\theoremstyle{plain}
\newtheorem*{theoremnonum}{Theorem}
\newtheorem*{corollarynonum}{Corollary}

\theoremstyle{definition}

\theoremstyle{plain}
\newtheorem{theorem}{Theorem}
\newtheorem{prop}[theorem]{Proposition}
\newtheorem{corollary}[theorem]{Corollary}
\newtheorem{lemma}[theorem]{Lemma}
\newtheorem{claim}[theorem]{Claim}

\newtheorem{setup}[theorem]{Setup}

\theoremstyle{definition}
\newtheorem{definition}[theorem]{Definition}

\newtheorem{example}[theorem]{Example}

\theoremstyle{remark}
\newtheorem{remark}[theorem]{Remark}

\numberwithin{theorem}{section}

\newcommand{\bP}{\mathbb{P}}

\newcommand{\cO}{\mathcal{O}}

\newcommand{\cJ}{\mathcal{J}}

\newcommand{\cI}{\mathcal{I}}

\newcommand{\fb}{\mathfrak{b}}

\newcommand{\set}[1]{\left\{#1\right\}}

\newcommand{\Q}{\mathbb Q}

\newcommand{\C}{\mathbb C}

\newcommand{\Sing}{\textrm{Sing }}

\newcommand{\Supp}{\textnormal{Supp}}
\newcommand{\ord}{\textrm{ord}}

\newcommand{\adj}{\textnormal{adj}}

\title{Inversion of subadjunction and multiplier ideals}
\author{Eugene Eisenstein}
\address{Department of Mathematics, University of Michigan, Ann Arbor, MI 48109, USA}
\email{eisenst@umich.edu}
\thanks{The author was partially supported by an NSERC Postgraduate Scholarship and the Shields fellowship.}
\date{\today}

\begin{document}

\begin{abstract} We present a generalization of the multiplier ideal version of inversion of adjunction, often known as the restriction theorem, to centers of arbitrary codimension. We approach inversion of adjunction from the subadjunction point of view. Let $X$ be a smooth complex projective variety, let $Z \subseteq X$ be an exceptional log-canonical center of an effective $\Q$-divisor $\Delta$ on some dense open subset of $X$ that contains the generic point of $Z$, let $A$ be ample, let $0 < \epsilon \ll 1$ be small and let $\nu : Z^n \to Z$ be the normalization. Using ideas of Kawamata, we construct a subadjunction formula $\nu^*(K_X + \Delta + \epsilon A) \sim_\Q K_{Z^n} + \Delta_{Z^n}$ for certain special choices of $\Delta_{Z^n}$. We also define an adjoint ideal $\adj_Z(X, \Delta)$ that measures how non-klt $(X, \Delta)$ is outside the generic point of $Z$. With this setup, we prove that $\cJ(Z^n, \Delta_{Z^n})$ is contained in the conductor of $\nu$ and, with this identification of $\cJ(Z^n, \Delta_{Z^n})$ with an ideal on $Z$, we show that $\adj_Z(X, \Delta) \cdot \cO_Z = \cJ(Z^n, \Delta_{Z^n})$. Our theorem extends Kawamata's subadjunction theorem and implies that $(Z^n, \Delta_{Z^n})$ is klt if and only if $Z$ is an exceptional log-canonical center of $(X, \Delta)$.
\end{abstract}

\maketitle

\section{Introduction}

Let $X$ be a smooth projective variety over $\C$, let $\Delta$ be a $\Q$-divisor on $X$, and let $H$ be a smooth irreducible divisor on $X$ not contained in the support of $\Delta$. Recall that we say that $\Delta$ is klt if for every divisor $E$ over $X$ the discrepancy of $\Delta$ at $E$ is strictly greater than $-1$, plt along $H$ if this is true for every divisor with from $H$, and log-canonical if the discrepancy of $\Delta$ at every divisor over $X$ is $\geq -1$. We can think of these three notions as various grades of singularity of $\Delta$, in order from less singular to more singular.

Also recall that we can measure the failure of $\Delta$ to be klt with a multiplier ideal $\cJ(X, \Delta)$ defined as follows. Let $\pi : Y \to X$ be a log-resolution of $(X, \Delta)$ and set
\[\cJ(X, \Delta) = \pi_*\cO_Y(\lceil K_Y - \pi^*(K_X + \Delta) \rceil ).\]
The deeper the ideal the worse the failure. We can similarly define more general adjoint ideals $\adj_H(X, \Delta)$ that measure the failure of $\Delta$ to be plt near $H$ via the formula
\[\adj_H(X, \Delta) = \pi_*\cO_Y(\lceil K_Y - \pi^*(K_X + \Delta) \rceil + H')\]
where $H'$ is the strict transform of $H$ and $\pi : Y \to X$ is a log-resolution of $(X, \Delta + H)$.

In this paper we will concentrate on inversion of adjunction, an important and much studied tool in birational geometry. We are especially interested in the adjoint ideal version:
\[\adj_H(X, \Delta) \cdot \cO_H = \cJ(H, \Delta_H)\]
Whenever $H \not\subseteq \Supp(\Delta)$. This in particular implies the geometrically appealing statement that $(H, \Delta_H)$ is klt if and only if $(X, \Delta)$ is plt near $H$, that is, plt and klt singularities do not see transversality problems. We even have a short exact sequence
\[0 \to \cJ(X, \Delta + H) \to \adj_H(X, \Delta) \to \cJ(H, \Delta_H) \to 0.\]
This description of the kernel of the restriction map is an important ingredient in the algebraic approach to the famous extension theorem of Siu, see for example \cite{HMBounded}, \cite{KawamataExtension}, \cite{TakayamaBounded}, as well as \cite{RobNotes}.

It is known that the requirement that $X$ and $H$ be smooth can be relaxed. We wish to ask a question in a different direction: does $H$ have to be a divisor? If we replace $H$ by $Z \subseteq X$ an arbitrary subvariety, is there then a statement? If $Z$ is a complete intersection, or even locally a complete intersection, the correct statement easily follows from the restriction theorem stated above. The work of Takagi in \cite{TakagiAdjoint} and an earlier paper \cite{MeTakagi} by the author investigate a statement where $Z$ is only $\Q$-Gorenstein. Unfortunately, the $\Q$-Gorenstein condition can be restrictive in general. Can inversion of adjunction be generalized further?

We propose to view the question from the point of view of subadjunction. We will call a subvariety $Z \subseteq X$ a \emph{generically exceptional log-canonical center} of $(X, \Delta)$ if, for every log-resolution $\pi : Y \to X$ of $(X, \Delta)$, there is a unique divisor $E$ with discrepancy $-1$ so that $\pi(E) = Z$. $Z$ may not be normal, so let $\nu : Z^n \to Z$ be the normalization. Let $A$ be an ample divisor on $X$ and let $0 < \epsilon \ll 1$ be a small positive rational number. We can tautologically write
\[\nu^*(K_X + \Delta + \epsilon A)_Z \sim_\Q K_{Z^n} + \Delta_{Z^n}.\]

Kawamata's celebrated subadjunction theorem says that, if $Z$ is minimal with respect to inclusion among the log-canonical centers of $(X, \Delta)$, then $Z$ is normal and $\Delta_Z$ can be chosen to be klt (see Section 8 in \cite{CortiFlips} or the original papers \cite{KawamataSubadj1} and \cite{KawamataSubadj2}). We will check in Theorem \ref{subadj} that, if $Z$ is an arbitrary generically exceptional log-canonical center of $(X, \Delta)$ then Kawamata's construction goes through to give a class of special choices for the error term $\Delta_{Z^n}$. Any $\Delta_{Z^n}$ constructed in this way is called a Kawamata different. In Definition \ref{suitabledifferent} we define the notion of a generic Kawamata different.

For any pair $(X, \Delta)$ with a generically exceptional log-canonical center $Z \subseteq X$ we can define an adjoint ideal $\adj_Z(X, \Delta)$ as follows. Let $\pi : Y \to X$ be a log-resolution of $(X, \Delta)$ and let $E$ be the unique divisor with discrepancy $-1$ that lies over $Z$. Then
\[\adj_Z(X, \Delta) = \pi_* \cO_Y( \lceil K_Y - \pi^*(K_X + \Delta) + E\rceil).\]
This ideal measures the failure of $(X, \Delta)$ to be klt outside the generic point of $Z$. With this setup, our main theorem relates the singularities of $(Z^n, \Delta_{Z^n})$ with those of $(X, \Delta)$ as follows.

\begin{theoremnonum}[Theorem \ref{inversionofsubadj}] Let $X$ be a smooth complex projective variety and let $\Delta$ be an effective $\Q$-divisor on $X$. Suppose that $Z$ is a generically exceptional log-canonical center of $(X, \Delta)$. Let $\nu : Z^n \to Z$ be the normalization of $Z$. Let $\Delta_{Z^n}$ be a generic Kawamata different for $Z$.

Then:
\begin{enumerate}
\item $\cJ(Z^n, \Delta_{Z^n})$ is contained in the conductor ideal of $\nu$.
\item The conductor is also an ideal on $Z$ and so $\cJ(Z^n, \Delta_{Z^n})$ is naturally an ideal on $Z$. With this identification, we have that
\[\adj_Z(X, \Delta) \cdot \cO_Z = \cJ(Z^n, \Delta_{Z^n}).\]
\item We have a natural exact sequence
\[0 \to \cJ(X, \Delta) \to \adj_Z(X, \Delta) \to \cJ(Z^n, \Delta_{Z^n}) \to 0.\]
\end{enumerate}
\end{theoremnonum}

The proof proceeds by constructing an appropriate resolution and pushing down the exact sequence that defines the adjoint ideal. We apply a technical lemma about local sections of certain sheaves, together with local vanishing, to prove that the term on the right-hand side of the short exact sequence as indeed an ideal whose local sections are precisely those satisfying some order of vanishing criteria. This allows us to reduce to a computation with discrepancies. We then carefully analyze these discrepancies with the higher dimensional Kodaira canonical bundle formula (see Theorem \ref{hodgethy} for the statement).

It is not hard to see that if $(X, \Delta')$ is a log-canonical pair then any subvariety $Z \subseteq X$ is a generically exceptional log-canonical center of some other $\Delta$ (see Lemma \ref{anyZisgenelcc}). Thus, our theorem is a form of inversion of adjunction that applies to an arbitrary subvariety of any smooth $X$.

This theorem has a number of immediate corollaries, including Kawamata's original theorem.

\begin{corollarynonum} All generic Kawamata differents are effective. All generic Kawamata differents have the same multiplier ideal.
\end{corollarynonum}

\begin{corollarynonum}[Kawamata subadjunction] If $\Delta$ is log-canonical and $Z$ is an exceptional log-canonical center of $\Delta$ then $Z$ is normal and any generic Kawamata different is effective and klt.
\end{corollarynonum}

We emphasize that our result uses Kawamata's methods and should not be regarded as a new proof of his subadjunction theorem.

\begin{corollarynonum}[Naive inversion of subadjunction] Suppose $Z$ is an exceptional log-canonical center of $\Delta$ in a neighborhood of the generic point of $Z$. Then any generic Kawamata different is klt on $Z^n$ if and only if $\Delta$ is log-canonical and $Z$ is a minimal log-canonical center of $\Delta$.
\end{corollarynonum}

We also easily deduce a Kawamata-Viehweg type vanishing theorem for our adjoint ideal (see Corollary \ref{KVadj}).

Earlier work on this question includes \cite{AmbroInversion}, \cite{DanoKimExtension} and \cite{KarlAdjunction}. In \cite{AmbroInversion} the author considers fiber spaces $E \to Z$ and base change along birational morhpisms $Z' \to Z$, in our case the morphisms involved may not be fiber products. In \cite{DanoKimExtension} the author proves an Ohsawa-Takegoshi type theorem, with estimates, for maximal log-canonical centers. This is analogous to our Kawamata-Viehweg type vanishing theorem but more powerful.  In \cite{KarlAdjunction} the author considers the positive characteristic question, using Frobenius splitting criteria as analogs of log-canonical centers.

We would like to thank L. Ein, K. Schwede and C. Xu for valuable discussions. We are very grateful to R. Lazarsfeld and M. Musta\c t\u a for their help and support.

In this paper we will work exclusively over $\C$.

\section{Generically exceptional log-canonical centers and the adjoint ideal}

We briefly recall the definition and basic properties of exceptional centers. First, we fix notation.

\begin{definition} Let $X$ be a quasi-projective variety and let $\Delta$ be a Weil $\Q$-divisor on $X$. We will say that $(X, \Delta)$ is a \emph{pair} if $X$ is normal, $\Delta$ is effective, and $K_X + \Delta$ is $\Q$-Cartier.
\end{definition}

\begin{definition}\label{lccenterdefn} Let $(X, \Delta)$ be a log-canonical pair. A subvariety $Z \subseteq X$ is called a log-canonical center if there exists a log-resolution $\pi : Y \to X$ of $\Delta$ and a divisor $E \subseteq Y$ with discrepancy $-1$ so that $\pi(E) = Z$.
\end{definition}

The following standard theorem is a crucial part of the theory.

\begin{theorem} Let $(X, \Delta)$ be a log-canonical pair. There exists a log-resolution $\pi : Y \to X$ of $\Delta$ so that all log-canonical centers of $\Delta$ are realized by $\pi$, in other words, for every subvariety $Z \subseteq X$ that is a log-canonical center of $\Delta$ there is a divisor $E \subseteq Y$ with discrepancy $-1$ so that $Z = \pi(E)$.
\end{theorem}
\begin{proof} The proof of Corollary 2.31 in \cite{KollarAndMori} shows that the number of log-canonical centers is finite. Once this is known the theorem is obvious.
\end{proof}

For our purposes it will be important to define the following special cases of log-canonical centers.

\begin{definition}\label{exccenterdefn} Let $(X, \Delta)$ be a log-canonical pair and let $\pi : Y \to X$ be a log-resolution of $\Delta$ realizing all of its log-canonical centers. Let $Z \subseteq X$ be a log-canonical center of $\Delta$.
\begin{enumerate}
\item $Z$ is a minimal log-canonical center if $Z$ is a minimal element of the set of log-canonical centers of $\Delta$ with respect to inclusion.
\item $Z$ is an exceptional log-canonical center if $Z$ is minimal and the divisor $E \subseteq Y$ with $\pi(E) = Z$ and discrepancy $-1$ is unique.
\item $Z$ is generically an exceptional log-canonical center if there is a dense open subset $U \subseteq X$ containing the generic point of $Z$ so that $Z \cap U$ is an exceptional log-canonical center of $(U, \Delta_U)$. In other words, the divisor $E \subseteq Y$ with $\pi(E) = Z$ and discrepancy $-1$ is unique but $Z$ may not be a minimal log-canonical center.
\end{enumerate}
\end{definition}

\begin{example} Suppose that $X$ is smooth and $\Delta$ is a reduced simple normal crossings divisor, say
\[\Delta = \sum E_i.\]
Then $(X, \Delta)$ is log-canonical. The log-canonical centers of $\Delta$ are simply intersections of the $E_i$. The minimal centers are the intersections of maximal subcollections $E_{i_\alpha}$ so that
\[Z = \bigcap_{i_\alpha} E_{i_\alpha}\]
is not empty. Every intersection of any subcollection of the $E_i$ is generically exceptional. All minimal centers are exceptional.
\end{example}

\begin{example}\label{genlccenters} Let $X = \C^2$ and let $\Delta$ be the cusp, that is, the image of the morphism $\C \to \C^2$ that sends $t \mapsto (t^2, t^3)$. Then $\Delta' = c \cdot \Delta$ is log-canonical when $c = 5/6$. The origin is an exceptional log-canonical center of $(X, \Delta')$. The cusp itself is a generically exceptional log-canonical center of $(X, \Delta)$ even though $(X, \Delta)$ is not log-canonical.
\end{example}

As we show now, if $(X, \Delta')$ is a log-canonical pair then any subvariety $Z \subseteq X$, not necessarily normal, is a generically exceptional log-canonical center of some other pair structure $(X, \Delta)$.

\begin{lemma}\label{anyZisgenelcc} Let $(X, \Delta')$ be a log-canonical pair and let $Z \subseteq X$ be an arbitrary subvariety. Then there exists a Weil $\Q$-divisor $\Delta$ on $X$ so that $Z$ is a generically exceptional log-canonical center of $(X, \Delta)$.
\end{lemma}
\begin{proof} Let $A$ be a divisor so ample that $\cO_X(A) \otimes \cI_Z$ is globally generated. Then we can choose an $s$ and choose divisors
\[H_1, \ldots, H_s \in |\cO_X(A) \otimes \cI_Z|\]
so that, if we set
\[\Delta_0 = (H_1 + \cdots + H_s)\]
then $\Delta' + \Delta_0$ is log-canonical outside $Z$ and is not klt at the generic point of $Z$. Take a log-resolution $f : Y \to X$ of $(X, \Delta' + \Delta_0)$. Consider, for some $c \in \Q_{>0}$ to be determined, the divisor
\[K_Y - f^*(K_X + \Delta' + c \cdot \Delta_0)= K_Y - f^*(K_X + \Delta') - c f^*(H_1 + \cdots H_s) = \sum_i b_i E_i.\]
Since $\Delta' + \Delta_0$ was log-canonical outside $Z$ but not klt at the generic point of $Z$ we can choose a $c$ with $0 < c \leq 1$ so that $\Delta' + c \cdot \Delta_0$ is log-canonical outside $Z$, all $b_i$ with $E_i$ dominating $Z$ are $\leq -1$ and some $b_i$ with $E_i$ dominating $Z$ is equal to $-1$. Then the locus where $\Delta' + c \cdot \Delta_0$ is not log-canonical does not contain the generic point of $Z$.

We have shown that there exists an open subset $U \subseteq X$ and a pair structure $(X, \Delta' + c \cdot \Delta_0)$ so that $Z \cap U$ is a minimal log-canonical center of $(U, (\Delta' + c \cdot \Delta_0)_U)$. We can then apply the tie-breaking procedure as in, for example, Proposition 8.7.1 in \cite{CortiFlips}, to get an $(X, \Delta)$ with $Z \cap U$ an exceptional log-canonical center of $(U, \Delta_U)$, that is, $Z$ is a generically exceptional log-canonical center of $(X, \Delta)$.
\end{proof}

For the rest of this paper we fix the following setup.

\begin{setup}\label{setupinvadj} Let $X$ denote a smooth projective variety and $\Delta$ an effective $\Q$-divisor on $X$. Let $Z \subseteq X$ be a subvariety that is a generically exceptional log-canonical center of $(X, \Delta)$.
\end{setup}

With this setup one can define the following variant of adjoint ideals.

\begin{definition}\label{adjdefn} Let $\Delta$ be a $\Q$-divisor as in Setup \ref{setupinvadj}. Let $g : X' \to X$ be a log-resolution of $\Delta$. Define
\[\adj_Z(X, \Delta) = g_* \cO_{X'}\left(\left\lceil K_{X'/X} - g^*\Delta\right\rceil + E\right)\]
where $E$ is the unique divisor dominating $Z$ with discrepancy $-1$.
\end{definition}

\begin{prop} This ideal does not depend on the choice of $g$.
\end{prop}
\begin{proof} Consider a sequence of birational maps
\[\xymatrix{
X'' \ar[r]^f\ar[dr]_\pi & X' \ar[d]^g \\
& X
}\]
with $g$ a log-resolution of $\Delta$ and $f$ a log-resolution of $g^*(\Delta) + \textnormal{Exc}(g)$. Note that $\pi$ is a log-resolution of $\Delta$. As in Theorem 9.2.20 in \cite{LazarsfeldPositivityII}, it is enough to show that $g$ and $\pi$ compute the same ideal.

By the projection formula, it is enough to show that, if $E^\pi$ and $E^g$ are the exceptional divisors of discrepancy $-1$ dominating $Z$ on $X''$ and $X'$ respectively, then
\begin{equation}\label{etsindep}
\left\lceil K_{X''/X} - \pi^*\Delta \right\rceil + E^\pi = f^*\left(\left\lceil K_{X'/X} - g^*\Delta \right\rceil + E^g\right) + B
\end{equation}
with $B$ effective and $f$-exceptional.

First, we write
\[\lceil K_{X''/X} - \pi^*\Delta \rceil = K_{X''/X'} - \lfloor -f^*(K_{X'/X} - g^*\Delta) \rfloor.\]
Set
\[D = K_{X'/X} - g^*\Delta.\]
Part (3) of Corollary 2.31 in \cite{KollarAndMori} applies to our situation and says that
\[a(E; X', \lceil D \rceil - D) > -1.\]
It follows that the divisor
\[\lceil K_{X''/X} - \pi^*\Delta \rceil - f^* \lceil D \rceil = K_{X''/X'} - \lfloor -f^*(K_{X'/X} - g^*\Delta) \rfloor - f^* \lceil D \rceil\]
is effective. In other words,
\begin{equation}\label{robindep}
\left\lceil K_{X''/X} - \pi^*\Delta \right\rceil = f^*\left\lceil K_{X'/X} - g^*\Delta \right\rceil + B'
\end{equation}
with $B'$ effective and $f$-exceptional. On the other hand, by the definition of the discrepancy we have that
\[\ord_{E^\pi}\left(\left\lceil K_{X''/X} - \pi^*\Delta \right\rceil\right) = \ord_{E^g}\left(\left\lceil K_{X'/X} - g^*\Delta \right\rceil\right) = -1.\]
It follows that
\begin{equation}\label{myindep}
\ord_{E^\pi}\left(\left\lceil K_{X''/X} - \pi^*\Delta \right\rceil + E^\pi\right) = \ord_{E^g}\left(\left\lceil K_{X'/X} - g^*\Delta \right\rceil + E^g \right) = 0.
\end{equation}
Thus, (\ref{robindep}) shows that (\ref{etsindep}) holds for all divisors except $E^\pi$ and (\ref{myindep}) shows that (\ref{etsindep}) holds for $E^\pi$. This gives (\ref{etsindep}).
\end{proof}

\section{Kawamata subadjunction and the Kawamata different}

Using a modification of the proof of Kawamata's subadjunction theorem we will now prove a subadjunction formula for generically exceptional log-canonical centers. We will follow the exposition of Koll\'ar in section 8 of \cite{CortiFlips} with some minor modifications that apply to our situation. First, we state the crucial fact from Hodge theory that everything depends on.

\begin{theorem}[Theorem 8.5.1 in \cite{CortiFlips}]\label{hodgethy} Let $E$ and $W$ be smooth projective varieties and let $f : E \to W$ be a dominant morphism. Let $F$ be a general fiber of $f$. Let $R$ be a $\Q$-divisor on $E$ and let $B$ be a reduced divisor on $W$ so that:
\begin{enumerate}
\item $K_E + R \sim_\Q f^*(\text{some divisor on } W)$,
\item The Kodaira dimension of the divisor $K_F + R_F$ on $F$ is equal to $0$,
\item $f : E \to W$, $R$ and $B$ satisfy the \emph{standard normal crossings assumptions}:
\begin{enumerate}
\item $E$ and $W$ are smooth (as assumed in our statement),
\item $R + f^*B$ and $B$ are simple normal crossings,
\item $f$ is smooth over $W \setminus B$, and
\item if $F'$ is a fiber of $f$ over a point $p \in W \setminus B$ then $R_{F'}$ is klt.
\end{enumerate}
\end{enumerate}
Then we can write
\[K_E + R \sim_\Q f^*(K_W + J(E/W, R) + B_R)\]
where:
\begin{itemize}
\item $J(E/W, R)$ is a divisor defined only up to linear equivalence. It is the so-called \emph{moduli part}. It depends only on $(F, R_F)$ and on $W$. Under our standard normal crossings assumptions it is nef.
\item $B_R$ is a $\Q$-divisor that is uniquely determined once we fix the divisor $B$. It is defined by the following condition: it is the unique $\Q$-divisor for which there is a codimension $\geq 2$ subset $S \subset W$ such that
\begin{enumerate}
\item $(E \setminus f^{-1}(S), R + f^*(B - B_R))$ is log-canonical, and
\item every irreducible component of $B$ is dominated by a log-canonical center of $(E, R + f^*(B - B_R))$.
\end{enumerate}
\end{itemize}
\end{theorem}

\begin{remark} In fact, our statement is less general than the statement of Koll\'ar in \cite{CortiFlips}. Because of this, condition (1) implies our condition (2). We included condition (2) in order to follow the exposition of Koll\'ar.
\end{remark}

We now come to one of the main ideas of our point of view on the subadjunction theorem. Recall the statement of Kawamata's subadjunction theorem: let $A$ be an ample divisor and let $0 < \epsilon \ll 1$ be a small rational number. Suppose that $Z \subseteq X$ is an exceptional log-canonical center of $\Delta$. Then $Z$ is normal and we can choose a Weil $\Q$-divisor $\Delta_Z$ on $Z$ so that
\[(K_X + \Delta + \epsilon A)_Z \sim_\Q K_Z + \Delta_Z\]
with $\Delta_Z$ klt. In our context, we can regard Kawamata's theorem as saying that the error term $\Delta_Z$ is small in an appropriate sense.

What if $Z$ is only generically exceptional? Then $Z$ may not be normal, so let $\nu : Z^n \to Z$ be the normalization morphism. We can always tautologically write
\[\nu^*(K_X + \Delta + \epsilon A)_Z \sim_\Q K_{Z^n} + \Delta_{Z^n}\]
for lots of choices of $\Delta_{Z^n}$. In fact, there are some natural choices of $\Delta_{Z^n}$ that are suggested by Kawamata's theorem. We construct these choices now.

\begin{theorem}\label{subadj} Let $X$ be a smooth projective variety, let $A$ be an ample divisor on $X$ and let $\epsilon > 0$. Let $\Delta$ be a $\Q$-divisor on $X$ and suppose that $Z$ is a generically exceptional log-canonical center of $(X, \Delta)$. Let $\nu : Z^n \to Z$ be the normalization of $Z$. With the above setup there exists an explicitly constructed $\Q$-divisor $\Delta_{Z^n}$ on $Z^n$ that we will call the \emph{Kawamata different}, so that
\begin{enumerate}
\item $K_{Z^n} + \Delta_{Z^n}$ is $\Q$-Cartier,
\item $\nu^*(K_X + \Delta + \epsilon A)_Z \sim_\Q K_{Z^n} + \Delta_{Z^n}$.
\end{enumerate}
\end{theorem}
\begin{proof} We begin with the following claim.

\begin{claim}\label{funkyresolution} There exists a log-resolution $g : Y' \to X$ of $\Delta$ with the following properties. Let $E$ be the unique divisor on $Y'$ of discrepancy $-1$ lying over $Z$ and let $g_E : E \to Z$ be the restriction of $g$. Then we can arrange for $g$ to factor as in the following diagram
\begin{equation*}
\xymatrix{
E \ar@/_2pc/[dd]_{g_E} \ar[d]^{f_E} &\subseteq& Y' \ar[d]_f \ar@/^2pc/[dd]^g\\
W \ar[d]^{\pi_E} &\subseteq& Y \ar[d]_\pi\\
Z &\subseteq& X
}
\end{equation*}
so that, if we write
\[K_{Y'} + E + \Delta' \sim_\Q g^*(K_X + \Delta)\]
with $g_*\Delta' = \Delta$ and let $R = \Delta'_E$ (this divisor has simple normal crossings support), then $W$ carries a divisor $B$ that, together with $R$ and $f$, satisfies the \emph{standard normal crossings assumptions} of Theorem \ref{hodgethy}.
\end{claim}

First, we will finish the proof of Theorem \ref{subadj} assuming the truth of the claim. Note that, since $Z$ is generically an exceptional log-canonical center of $\Delta$, $R$ is klt on a generic fiber of $g_E$.

Using Theorem \ref{hodgethy}, we obtain a divisor $B_R$ supported on $B$ so that
\[K_E + R \sim_\Q f_E^*(K_W + J(E/W, R) + B_R).\]
Since $A$ is ample and $J(E/W, R)$ is nef, the sum $J(E/W, R) + \epsilon \pi_E^*A$ is big and nef and so is $\Q$-equivalent to some effective divisor $J_\epsilon$. But recall that $K_E + R \sim_\Q g_E^*(K_X + \Delta)_Z$. It follows that
\[K_{W} + J_\epsilon + B_R \sim_\Q \pi_E^*(K_X + \Delta + \epsilon A)_Z.\]
Now, $\pi_E$ must factor through the normalization $\nu : Z^n \to Z$; write $\pi_E = \nu \circ h$ for this factorization. Pushing the above formula forward along $h$ yields
\[\nu^*(K_X + \Delta + \epsilon A)_Z \sim_\Q K_{Z^n} + h_* (J_\epsilon + B_R).\]
Set $\Delta_{Z^n} = h_* (J_\epsilon + B_R)$.
\end{proof}

\begin{proof}[Proof of Claim \ref{funkyresolution}] For clarity, we will construct the required resolution in several steps.

\emph{Step 1:} Begin with any log-resolution of $\Delta$, say $g : X_1 \to X$, and let $R$ and $E$ be as in the statement of the claim. Take any reduced divisor $B_0$ so that $\Supp(B_0)$ contains $\Sing(Z)$ and the locus of points at which $g$ is not smooth or $R$ and $F_p$ are not simple normal crossings, where $F_p$ is the fiber of $g_E : E \to Z$ over $p$.

\emph{Step 2:} Take a smooth blow-up $\pi : Y \to X$ that does not blow up the generic point of $Z$ and that induces a birational morphism $\pi_E : W \to Z$ with $W$ smooth and $B = \textnormal{red}(\pi_E^* B_0)$ simple normal crossings. Note that, currently, $\pi$ is related to $g$ only because $B$ depends on $g$. We have so far the following diagram
\[\xymatrix{
W \ar@{^{(}->}[r]\ar[dr]_{\pi_E} & Y \ar[dr]_\pi & & X_1 \ar[dl]^g\\
& Z \ar@{^{(}->}[r] & X
}\]

\emph{Step 3:} In the above diagram we can identify two opportunities for a fiber product:
\[\xymatrix{
& E \ar[d]\\
W \ar[r] & Z
}
\ \ \ \ \ \ \ 
\xymatrix{
& X_1 \ar[d]\\
Y \ar[r] & X
}
\]
Let $E' \subseteq E \times_Z W$ be the component of the fiber product dominating $W$ and let $X' \subseteq X_1 \times_X Y$ be the component of the fiber product containing $E'$. Note that the blow-up $X' \to X_1$ is an isomorphism outside $B$ so that, outside $B$, we have that $E'$ is isomorphic to $E$ and $X'$ is isomorphic to $X_1$. In particular, $X'$ is a log-resolution of $\Delta$ outside $B$.

\emph{Step 4:} We have the following diagram
\[\xymatrix{
E' \ar@{^{(}->}[r]\ar[d] & X' \ar[d] \\
W \ar@{^{(}->}[r]\ar[d] & Y \ar[d] \\
Z \ar@{^{(}->}[r] & X
}\]
with $X' \to X$ isomorphic to $X_1 \to X$ outside $B$. Complete the diagram to a new diagram
\[\xymatrix{
E \ar@{^{(}->}[r]\ar[d] & Y' \ar[d] \\
E' \ar@{^{(}->}[r]\ar[d] & X' \ar[d] \\
W \ar@{^{(}->}[r]\ar[d] & Y \ar[d] \\
Z \ar@{^{(}->}[r] & X
}\]
where the morphism $g : Y' \to X' \to X$ has the following properties. Let $k : Y' \to X'$ be the factoring morphism. We require that $g$ be a log-resolution of $\Delta$ and that, if we replace $E$ by its strict transform under $k$, $E$ becomes smooth. Let $R' = k^*R$, write
\[K_{Y'} + E + \Delta' \sim_\Q g^*(K_X + \Delta)\]
with $g_*\Delta' = \Delta$ and replace $R$ with $R = \Delta'_E$. The next property for which we ask is that the exceptional set of $k$ should have simple normal crossings. Note that then $R$ differs from $R'$ by divisors that are restrictions to $E$ of divisors exceptional for $k$. So, we may also require that $R + f^*B$ be simple normal crossings.

Since $X' \to X$ is isomorphic to $X_1 \to X$ outside $B$, $k$ may be chosen to only blow up centers whose images on $X$ are contained in $B$. It follows that these choices construct the diagram in the statement of the claim with the given initial choice of $g : X_1 \to X$ and $B$.
\end{proof}

The following condition on $\Delta_{Z^n}$ should be thought of as saying that $\Delta_{Z^n}$ is sufficiently generic.

\begin{definition}\label{suitabledifferent} Notation as in the previous theorem. We say that a Kawamata different $(Z^n, \Delta_{Z^n})$ is \emph{generic} if, in addition to the requirements in Claim \ref{funkyresolution}, the following are satisfied.
\begin{itemize}
\item The map $\pi_E$ is sufficiently high - the Rees valuations of $\adj_Z(X, \Delta) \cdot \cO_{Z^n}$ are extracted by $\pi_E$,
\item $J_\epsilon$ is general - $\lfloor J_\epsilon + B_R \rfloor = \lfloor B_R \rfloor$,
\item $B$ is sufficiently large - the components of $B$ include the Rees valuations of $\adj_Z(X, \Delta) \cdot \cO_{Z^n}$,
\end{itemize}
\end{definition}

\begin{remark} To achieve this we make our choices in Claim \ref{funkyresolution} as follows. We select $B$ large enough to contain the support of $\adj_Z(X, \Delta) \cdot \cO_{Z^n}$ and we select $\pi : Y \to X$ to factor through the blow-up of $\adj_Z(X, \Delta)$, this makes $\pi_E$ sufficiently high and $B$ sufficiently large. To make $J_\epsilon$ sufficiently general, we use the fact that $J_\epsilon$ is big and nef to choose it to be of the form $H + \epsilon C$, where $H$ is a general ample divisor, $C$ is effective and $\epsilon$ is sufficiently small. Note that all this may \emph{a priori} change $B_R$ and $\Delta_{Z^n}$.
\end{remark}

\section{Inversion of subadjunction}

We are now ready to state our main theorem.

\begin{theorem}[Inversion of subadjunction]\label{inversionofsubadj} Let $X$ be a smooth projective variety and let $\Delta$ be a $\Q$-divisor on $X$. Suppose that $Z \subseteq X$ is a generically exceptional log-canonical center of $(X, \Delta)$. Let $\nu : Z^n \to Z$ be the normalization of $Z$. Let $\Delta_{Z^n}$ be a generic Kawamata different as in Definition \ref{suitabledifferent}.

Since $K_{Z^n} + \Delta_{Z^n}$ is $\Q$-Cartier we may consider the multiplier ideal $\cJ(Z^n, \Delta_{Z^n})$ on $Z^n$. Then:
\begin{enumerate}
\item $\cJ(Z^n, \Delta_{Z^n})$ is contained in the conductor ideal of $\nu$.
\item The conductor is also an ideal on $Z$ and so $\cJ(Z^n, \Delta_{Z^n})$ can naturally be viewed as an ideal on $Z$. With this identification, we have that
\[\adj_Z(X, \Delta) \cdot \cO_Z = \cJ(Z^n, \Delta_{Z^n}),\]
\item We have a natural exact sequence
\[0 \to \cJ(X, \Delta) \to \adj_Z(X, \Delta) \to \cJ(Z^n, \Delta_{Z^n}) \to 0\]
of sheaves on $Z$.
\end{enumerate}
\end{theorem}

Before giving the proof of the theorem we recall the local vanishing theorem (see Theorem 9.4.1 in \cite{LazarsfeldPositivityII} for the proof), as well as record a simple but crucial lemma.

\begin{prop}[Local vanishing]\label{localvanishing} Let $f : Y' \to X$ be a proper birational morphism with $X$, $Y'$ projective varieties and $Y'$ smooth. Let $D$ be a $\Q$-divisor that has simple normal crossings support and is numerically equivalent to $K_{Y'} + f^* D'$ where $D'$ is any $\Q$-divisor on $X$. Then
\[R^i f_* \cO_{Y'}(\lceil D \rceil) = 0\]
for all $i > 0$.
\qed
\end{prop}

\begin{lemma}\label{localsections} Let $f : Y' \to X$ be a proper birational morphism between projective varieties and let $Z$ be a subvariety of $X$. Let $E \subseteq Y'$ be an irreducible divisor lying over $Z$. Let $f_E : E \to Z$ be the restriction of $f$ and let $D$ be a Cartier divisor on $Y'$ with $E \not\subseteq \Supp(D)$. Suppose that the natural map of sheaves
\[f_*\cO_{Y'}(D) \to f_{E,*}\cO_E(D_E)\]
induced by restriction of sections is surjective. Let $U$ be an open subset of $Z$. Then we can describe the sheaf $f_{E,*}\cO_E(D_E)$ by the rule
\[\Gamma(U, f_{E,*} \cO_E(D_E)) = \set{p \in \C(Z)\ |\ f_E^*(p) \in \Gamma(f_E^{-1}(U), \cO_E(D_E))}.\]
In other words, every rational function in the set $\Gamma(f_E^{-1}(U), \cO_E(D_E))$ is a pull-back of a rational function from $Z$.
\end{lemma}
\begin{proof}[Proof of Lemma] Let
\[\mathcal{S}(U) = \set{p \in \C(Z)\ |\ f_E^*(p) \in \Gamma(f_E^{-1}(U), \cO_E(D_E))}.\]
It is easy to see that this assignment, together with the obvious restriction maps, defines a sheaf $\mathcal{S}$ on $Z$ (even an $\cO_Z$-module). On the other hand, since $E \not\subseteq \Supp(D)$, we can define $f_*\cO_{Y'}(D) \cdot \cO_Z$ in the usual way since all rational functions in $\cO_{Y'}(D)$ are regular at the generic point of $E$ and obtain $\mathcal{S}$ more intrinsically.

By the definition of $\mathcal{S}$, there is a natural map of sheaves
\[\varphi : \mathcal{S} \to f_{E,*} \cO_E(D_E)\]
given by $p \mapsto f_E^*(p)$. We wish to show that $\varphi$ is an isomorphism. It is injective since $E$ dominates $Z$. Since both source and target are sheaves, if $\varphi$ is surjective as a map of sheaves then $\varphi_U : \Gamma(U, \mathcal{S}) \to \Gamma(U, f_{E,*} \cO_E(D_E))$ is an isomorphism for every open subset $U$ of $Z$.

Notice however that we can factor $\varphi$ as follows. Let $U$ be an open subset of $Z$, let $p \in \Gamma(U, \mathcal{S})$ and let $p'$ be any rational function on $X$ so that $p'_Z = p$. Then $f^*(p')_E = \varphi_U(p)$. But, since $E \not\subseteq \Supp(D)$, the map
\[f_*\cO_{Y'}(D) \to f_{E,*}\cO_E(D_E)\]
is nothing more than the map that, for an open subset $V$ of $X$ takes a rational function $p' \in \Gamma(V, f_*\cO_{Y'}(D))$ and maps it to $f^*(p')_E$. By hypothesis, this map is surjective as a map of sheaves. But this map clearly factors through $\varphi$ and so $\varphi$ is also surjective, as required.
\end{proof}

\begin{remark} Note that the conclusion of the lemma is equivalent to the statement that the natural ``base change'' map
\[f_* \cO_{Y'}(D) \cdot \cO_Z \to f_{E, *} \cO_E(D_E)\]
is an isomorphism. In fact, this is how the proof of the lemma proceeds. We will make use of this equivalent formulation.
\end{remark}

We now turn to the proof of Theorem \ref{inversionofsubadj}.

\begin{proof}[Proof of Theorem \ref{inversionofsubadj}] To make the proof more clear, we will proceed in several steps. As shorthand, set
\[\fb = \adj_Z(X, \Delta) \cdot \cO_Z.\]

\emph{Step 1:} We show that there is a natural exact sequence
\[0 \to \cJ(X, \Delta) \to \adj_Z(X, \Delta) \to \fb \to 0,\]
that $\fb$ is contained in the conductor of $\nu$, that it is integrally closed on $Z^n$, and we describe its local sections. We accomplish this by combining local vanishing and Lemma \ref{localsections} in the following manner. First, we construct the diagram of morphisms
\[
\xymatrix{
E \ar[dr]^{h_E}\ar[dd]_{g_E} \\
& Z^n \ar[dl]^{\nu_E} \\
Z
}
\xymatrix{
\subseteq \\
\subseteq \\
\subseteq
}
\xymatrix{
& Y' \ar[dl]_h\ar[dd]^g\\
X_0 \ar[dr]_\nu\\
& X
}
\]
using the following steps:
\begin{itemize}
\item Step 1: Let $\nu_E : Z^n \to Z$ is the normalization map. It is proper and birational and therefore it is given by the blowing up of some ideal sheaf $\cI$ on $Z$. Lift $\cI$ in an arbitrary manner to an ideal sheaf on $X$ and blow up this ideal sheaf to obtain $X_0$ and $\nu : X_0 \to X$. Thus, $X_0$ is reduced but possibly not normal.
\item Step 2: Complete $\nu$ to a log-resolution $g : Y' \to X$ of $\Delta$. Let $E$ be the unique divisor lying over $Z$ with discrepancy $-1$ and let $g_E : E \to Z$ be the restriction of $g$ to $E$.
\item Step 3: With these choices, $g_E$ factors through $\nu_E$ and $g$ factors through $\nu$. Let the factorizations be $g_E = \nu_E \circ h_E$ and $g = \nu \circ h$, here $h_E$ is the restriction of $h$ and $\nu_E$ is the restriction of $\nu$.
\end{itemize}

Now, let
\[D = K_{Y'/X} - g^*\Delta + E.\]
Consider the short exact sequence
\begin{equation}\label{ses1}
0 \to \cO_{Y'}(\lceil D - E \rceil) \to \cO_{Y'}(\lceil D \rceil) \to \cO_{Y'}(\lceil D \rceil)_E \to 0.
\end{equation}
of sheaves on $Y'$. Because of our assumption that $Z$ is a generically exceptional log-canonical center of $\Delta$, $E$ cannot be in the support of $D$. As we are in a simple normal crossings situation, we have
\[\cO_{Y'}(\lceil D \rceil)_E = \cO_{E}(\lceil D_E \rceil).\]
By the local vanishing theorem \ref{localvanishing} applied to (\ref{ses1}) and the morphism $g$, we get the short exact sequence
\[0 \to \cJ(X, \Delta) \to \adj_Z(X, \Delta) \to g_{E,*}\cO_{E}(\lceil D_E \rceil) \to 0.\]
Then Lemma \ref{localsections} says that the natural map
\[\fb := \adj_Z(X, \Delta) \cdot \cO_Z \to g_{E,*}\cO_E(\lceil D_E \rceil)\]
is an isomorphism.

We can also apply local vanishing to (\ref{ses1}) and the morphism $h$. We obtain that
\[R^1 h_* \cO_{Y'}(\lceil D - E \rceil) = 0.\]
Lemma \ref{localsections} again says that the natural map
\[h_* \cO_{Y'}(\lceil D \rceil) \cdot \cO_{Z^n} \to h_{E,*}\cO_E(\lceil D_E \rceil)\]
is an isomorphism. In particular, the sheaf $h_{E,*}\cO_E(\lceil D_E \rceil)$ is naturally a subsheaf of the function field of $Z$. But we have just seen that $g_{E,*}\cO_E(\lceil D_E \rceil) = \nu_{E,*}(h_{E,*}\cO_E(\lceil D_E \rceil))$ is an ideal of $\cO_Z$ (in a compatible sense) and so $\fb$ is contained in the conductor of $\nu$.

Note that this describes the local sections of $\fb$ as follows. Let $p \in \Gamma(U, \cO_Z)$ be a regular function on an open set $U$. Then $g_E^*(p) \in \Gamma(g^{-1}(U), \cO_{E}(\lceil D_E \rceil))$ if and only if $p \in \Gamma(U, \fb)$. In particular, since the membership criteria for $\fb$ are clearly given by valuations and $\fb$ is an ideal subsheaf of the sheaf of integrally closed rings $\cO_{Z^n}$, $\fb$ is integrally closed on $Z^n$.

We also emphasize that, since $\fb = \adj_Z(X, \Delta) \cdot \cO_Z$, the ideal $\fb$ does not depend on any choices of log-resolutions or Kawamata boundaries.

\emph{Step 2:} Next we make use of the fact, just proven, that $\fb$ is integrally closed in order to make our choice of log-resolution and other parameters for the rest of the proof. Let $R_i$ be the finite set of divisors over $Z$ that compute membership in the integrally closed ideal $\fb$, that is, the Rees valuations of $\fb$. As in Claim \ref{funkyresolution}, our diagram of morphisms will be as follows:
\[\xymatrix{
E \ar@/_2pc/[dd]_{g_E} \ar[d]^{f_E} \ar@{^{(}->}[r] & Y' \ar[d]_f \ar@/^2pc/[dd]^g\\
W \ar[d]^{\pi_E} \ar@{^{(}->}[r] & Y \ar[d]_\pi\\
Z \ar@{^{(}->}[r] & X
}\]
where:
\begin{itemize}
\item $g : Y' \to X$ is a log-resolution of $\Delta$, $E$ is the unique divisor of discrepancy $-1$ lying over $Z$, and $g_E : E \to Z$ is the restriction of $g$ to $E$.
\item $\pi_E : W \to Z$ is a proper birational morphism with simple normal crossings exceptional divisor, chosen so that it extracts the $R_i$. We choose $\pi : Y \to X$ to be a proper birational morphism that induces $\pi_E : W \to Z$ by restriction.
\item Using Claim \ref{funkyresolution}, we may additionally choose $g$ and $\pi_E$ in such way as to have a reduced divisor $B$ on $W$ with the properties that:
\begin{itemize}
\item $B$ satisfies the \emph{standard normal crossings assumptions} of Theorem \ref{hodgethy}. Denote by $B_R$ the divisor constructed from $B$ in Theorem \ref{hodgethy}.
\item $R_i \subseteq \Supp(B)$.
\end{itemize}
\item Again, $\pi_E : W \to Z$ factors through the normalization $\nu : Z^n \to Z$ and we write $\pi_E = \nu \circ h$ for the factorization.
\end{itemize}
Note that these conditions say that the resulting Kawamata different is generic in the sense of Definition \ref{suitabledifferent}. We define $E_i$ to be the components of $B$. 

Next, we adopt the notation from the proof of Kawamata's subadjunction theorem in Theorem \ref{subadj}. Notice that in fact
\[-(J_\epsilon + B_R) = K_W - h^*(K_{Z^n} + \Delta_{Z^n})\]
as $\Q$-divisors. Indeed, their non-exceptional parts are equal by definition and it follows that the exceptional parts are $\Q$-equivalent, hence equal. In particular, since $\Delta_{Z^n}$ is generic,
\[h_* \cO_W(\lceil -B_R \rceil) = \cJ(Z^n, \Delta_{Z^n}).\]

\emph{Step 3:} We finally compare $\fb$ and $\cJ(Z^n, \Delta_{Z^n})$. For each index $i$, let $F_i^\alpha$ be the divisors on $E$ that dominate $E_i$ (the indices $\alpha$ runs through depend on $i$). Note that we do \emph{not} claim that $\ord_{F_i^\alpha}(f^*B_R) = \ord_{F_i^\alpha}(R)$ for all $i$ and $\alpha$!

To make the comparison, recall from the definition of $B_R$ that
\begin{enumerate}
\item[(a)] For any irreducible divisor $G$ on $W$, $(E, R + f_E^*(B - B_R))$ is log-canonical in a neighborhood of the generic point of every component of $f_E^* G$ that dominates $G$,
\item[(b)] every component of $B$ is dominated by a log-canonical center of $(E, R + f_E^*(B - B_R))$.
\end{enumerate}
Since $R + f_E^*(B - B_R)$ is a simple normal crossings divisor by assumption, our choice of $B$ from step 2 and condition (a) say that
\[\ord_{F_i^\alpha} (R - f_E^*B_R) \leq 1 - \ord_{F_i^\alpha} (f_E^*B) \leq 0\]
for all $i$ and $\alpha$. This says that $\cJ(Z^n, \Delta_{Z^n}) = h_* \cO_W(\lceil -B_R \rceil) \subseteq \fb$.

For the reverse inequality, notice that condition (b) says that for every $i$ there is an $\alpha$ so that
\[\ord_{F_i^\alpha}(-R + f_E^*(B_R - B)) = -1.\]
So suppose that (locally) there were to exist an element $p \in \fb \setminus \cJ(Z^n, \Delta_{Z^n})$. Then, on the one hand, we have
\[\ord_{F_i^\alpha} (g_E^* p) \geq \ord_{F_i^\alpha} (R)\]
for all $i$ and $\alpha$. On the other hand, there must exist an index $i$ with
\[\ord_{E_i} (\pi_E^* p) < - \lceil - \ord_{E_i} (B_R) \rceil = \lfloor \ord_{E_i} (B_R) \rfloor.\]
Since the left hand side is an integer, this inequality is satisfied if and only if
\[\ord_{E_i} (\pi_E^* p) \leq \ord_{E_i} (B_R) - 1.\]
Pulling back we obtain, for this $i$ and all $F_i^\alpha$,
\[\ord_{F_i^\alpha} (g_E^* p) \leq \ord_{F_i^\alpha} (f_E^* (B_R - B)).\]
Putting the two inequalities together we see that, we must have
\[\ord_{F_i^\alpha}(R) \leq \ord_{F_i^\alpha}(g_E^* p) \leq \ord_{F_i^\alpha}(f_E^* (B_R - B)).\]
Then, for this $i$ and all $F_i^\alpha$, $\ord_{F_i^\alpha} (-R + f_E^*(B_R - B)) \geq 0$, a contradiction.
\end{proof}

\section{Corollaries}

This theorem has a number of immediate corollaries, including Kawamata's subadjunction statement as well as a naive version of inversion of subadjunction.

\begin{corollary} All generic Kawamata differents are effective. All generic Kawamata differents have the same multiplier ideal.
\end{corollary}
\begin{proof} By the theorem, $\cJ(Z^n, \Delta_{Z^n})$ is always an ideal. This is equivalent to the assertion that $\Delta_{Z^n}$ is effective. Also, $\cJ(Z^n, \Delta_{Z^n}) = \adj_Z(X, \Delta) \cdot \cO_Z$ and this latter ideal does not depend on the choice of $\Delta_{Z^n}$.
\end{proof}

\begin{example} This example can be found in \cite{AmbroAdjunctionConjecture}. Let $C \subseteq \bP^2$ be the curve defined by the equation $x^2 z - y^3 = 0$. The normalization of this curve is a $\bP^1$. Let $\nu : C^n \to C$ be the normalization map. Direct computation shows that
\[\nu^*(K_{\bP^2} + C)_C = K_{C^n} + 2p\]
with $p \in \bP^1$ a point. $C$ is, of course, a generically exceptional log-canonical center of $\Delta = C$.
\end{example}

\begin{example} Consider the twisted cubic $C \subseteq \bP^3$. There are two quadrics $H_1, H_2 \subseteq \bP^2$ with $H_1 \cap H_2 = C \cup L$ with $L$ a line at infinity. Setting $\Delta = H_1 + H_2$ we can check by direct computation that $C$ is a generically exceptional log-canonical center of $\Delta$, although it is not minimal - the point $C \cap L$ is also a log-canonical center. We can easily check that $(K_{\bP^3} + \Delta)_C$ is ample while $C$ is, of course, Fano, so the difference is ample and, in particular, effective.
\end{example}

\begin{corollary}[Kawamata subadjunction]\label{kawsubadj} If $\Delta$ is log-canonical and $Z$ is an exceptional log-canonical center of $\Delta$, then $Z$ is normal and any generic Kawamata different is effective and klt.
\end{corollary}
\begin{proof} If $\Delta$ is log-canonical and $Z$ is a minimal center then
\[\adj_Z(X, \Delta) = \cO_X.\]
It follows from Theorem \ref{inversionofsubadj} that $\cJ(Z^n, \Delta_{Z^n}) = \cO_{Z^n}$. But the theorem also tells us that $\cJ(Z^n, \Delta_{Z^n})$ is contained in the conductor of $\nu$. This conductor is therefore the unit ideal, that is, $Z$ is normal. Furthermore, the formula $\cJ(Z, \Delta_Z) = \cO_Z$ immediately implies that $\Delta_Z$ is effective and klt.
\end{proof}

\begin{corollary}[Naive inversion of subadjunction] Suppose $Z$ is a generically exceptional log-canonical center of $\Delta$. Then any generic Kawamata different is klt on $Z^n$ if and only if $\Delta$ is log-canonical and $Z$ is a minimal log-canonical center of $\Delta$.
\end{corollary}
\begin{proof} Since $\adj_Z(X, \Delta) \cdot \cO_Z = \cJ(Z^n, \Delta_{Z^n})$, the equivalence follows from checking when each side of this equation can be equal to $\cO_{Z^n}$.
\end{proof}

\begin{corollary}[Kawamata-Viehweg vanishing for $\adj_E(X, \Delta)$]\label{KVadj} Suppose that $Z$ is normal and $A$ is a big and nef $\Q$-divisor with $A_Z$ again big (in particular, if $Z \not\subseteq \mathbb{B}_+(A)$). Suppose that $L$ is a Cartier divisor with $A \equiv_\textnormal{num} L - \Delta$. Then
\[H^i(X, \cO_X(K_X + L) \otimes \adj_Z(X, \Delta)) = 0\]
for all $i > 0$.
\end{corollary}
\begin{proof} This follows immediately from Kawamata-Viehweg vanishing applied to the long exact sequence in cohomology that we get from the short exact sequence in Theorem \ref{inversionofsubadj}.
\end{proof}

\bibliography{paper}{}
\bibliographystyle{plain}

\end{document}